\documentclass[11pt,a4paper]{article}

\usepackage{epsf,epsfig,amsfonts,amsgen,amsmath,amstext,amsbsy,amsopn,amsthm
}
\usepackage{ebezier,eepic}
\usepackage{color}
\usepackage{multirow}
\newtheorem{theorem}{Theorem}

\newtheorem{lemma}{Lemma}
\newtheorem{false statement}{False statement}

\theoremstyle{definition}

\newtheorem{claim}{Claim}

\newtheorem{remark}[claim]{Remark}

\newtheorem{corollary}[claim]{Corollary}
\newtheorem{problem}{Problem}
\newtheorem{case}{Case}

\newcounter{mathitem}
  {\begin{list}{{$(\roman{mathitem})$}}{
   \setcounter{mathitem}{0}
   \usecounter{mathitem}
   \setlength{\topsep}{0pt plus 2pt minus 0pt}
   \setlength{\parskip}{0pt plus 2pt minus 0pt}
   \setlength{\partopsep}{0pt plus 2pt minus 0pt}
   \setlength{\parsep}{0pt plus 2pt minus 0pt}
   \setlength{\leftmargin}{35pt}
   \setlength{\itemsep}{0pt plus 2pt minus 0pt}}}
  {\end{list}}

\begin{document}

\title{\bf\Large Spectral radius and Hamiltonian properties of graphs\footnote{We found a gap in the
proof of Theorem 2 in previous versions including the published version. In this version,
we fill the gap and correct some typos as well.}}

\date{}

\author{Bo Ning$^{1}$\footnote{Corresponding author. E-mail address: ningbo\_math84@mail.nwpu.edu.cn}
and Jun Ge$^{2}$\footnote{E-mail address: mathsgejun@163.com}\\[2mm]
\small $^{1}$Department of Applied Mathematics, School of Science\\
\small  Northwestern Polytechnical University, Xi'an, Shaanxi 710072, P. R. China\\
\small $^{2}$School of Mathematical Sciences \\
\small Xiamen University, Xiamen, Fujian 361005, P. R. China}

\maketitle

\begin{abstract}
Let $G$ be a graph with minimum degree $\delta$. The spectral radius of $G$,
denoted by $\rho(G)$, is the largest eigenvalue of the adjacency matrix
of $G$. In this note we mainly prove the following two results.

(1) Let $G$ be a graph on $n\geq 4$ vertices with $\delta\geq 1$. If $\rho(G)> n-3$, then $G$ contains a Hamilton path unless $G\in\{K_1\vee (K_{n-3}+2K_1),K_2\vee 4K_1,K_1\vee (K_{1,3}+K_1)\}$.

(2) Let $G$ be a graph on $n\geq 14$ vertices with $\delta \geq 2$. If
$\rho(G)\geq \rho(K_2\vee (K_{n-4}+2K_1))$, then $G$ contains a
Hamilton cycle unless $G= K_2\vee (K_{n-4}+2K_1)$.

As corollaries of our first result, two previous theorems due to Fiedler \& Nikiforov
and Lu et al. are obtained, respectively. Our second result refines another previous theorem of Fiedler \& Nikiforov.


\end{abstract}

\section{Introduction}
Throughout this note, we use $G=(V(G),E(G))$ to denote a finite simple
undirected graph with vertex set $V(G)$ and edge set $E(G)$. Given a
graph $G$, we use $A$ to denote its adjacency matrix. Let $v_i\in V(G)$.
We denote by $d_i$ the \emph{degree} of $v_i$. Let $(d_1,d_2,...,d_n)$
be the \emph{degree sequence} of $G$, where $d_i$'s are in non-decreasing order.
The \emph{spectral radius} of $G$, denoted by $\rho(G)$, is the largest eigenvalues of $A$.
We denote by $\delta(G)$ or simply $\delta$ the \emph{minimum degree} of $G$.

Let $G_1=(V(G_1),E(G_1))$ and $G_2=(V(G_2),E(G_2))$ be two graphs. The \emph{union}
of $G_1$ and $G_2$, denoted by $G_1\cup G_2$, is the graph with vertex set
$V(G_1)\cup V(G_2)$ and edge set $E(G_1)\cup E(G_2)$. If $G_1$ and $G_2$ are
disjoint, then we call their union a \emph{disjoint union}, and denote it by
$G_1+G_2$. We denote the union of $k$ disjoint copies of a graph $G$ by $kG$.
The \emph{join} of two disjoint graphs $G_1$ and $G_2$, denoted by $G_1\vee G_2$,
is obtained from $G_1+G_2$ by joining each vertex of $G_1$ to each vertex of $G_2$.

In \cite{Brualdi_Solheid}, Brualdi and Solheid raised the following spectral problem.

\begin{problem}\label{pr1}
What is the maximum spectral radius of a graph $G$ on $n$ vertices belonging to a specified class of graphs?
\end{problem}

Recently, the following important type of problem (Brualdi-Solheid-Tur\'an type problem) has been extensively studied by many graph theorists.

\begin{problem}\label{pr2}
For a given graph $F$, what is the maximum spectral radius of a graph $G$ on $n$ vertices without subgraph isomorphic to $F$?
\end{problem}

Up to now, Problem \ref{pr2} has been considered for the cases that $F$ is a clique, an even or odd path (cycle) of given length, and a Hamilton path (cycle) \cite{Fiedler_Nikiforov,Lu_Liu_Tian,Nikiforov2,Nikiforov3,Yuan_Wang_Zhai,Zhai_Wang}.

In particular, sufficient spectral conditions for the existence of Hamilton paths and cycles receive extensive attention from many graph theorists. Fiedler and Nikiforov \cite{Fiedler_Nikiforov} gave tight sufficient conditions for the existence of Hamilton paths and cycles in terms of the spectral radius of graphs or the complement of graphs. Lu \cite{Lu_Liu_Tian} et al. studied sufficient conditions for Hamilton paths in connected graphs and Hamilton cycles in bipartite graphs in terms of the spectral radius of a graph. Some other spectral conditions for Hamilton paths and cycles in graphs have been given in \cite{Butler_Chung,Heuvel,Krivelevich_Sudakov,Mohar,Zhou}.

Since $\delta\geq 1$ ($\delta\geq 2$) is a trivial necessary condition when finding a Hamilton path (Hamilton cycle) in a given graph $G$, we make this assumption when finding spectral conditions for Hamilton paths (Hamilton cycles) of graphs throughout this note.

In this note, we mainly get two theorems.

\begin{theorem}\label{tNG1}
Let $G$ be a graph on $n\geq 4$ vertices with $\delta\geq 1$. Let $G^{(1)}_n=K_1\vee (K_{n-3}+2K_1)$. If $\rho(G)> n-3$, then $G$ contains a Hamilton path unless $G\in \{G^{(1)}_n,K_2\vee 4K_1,K_1\vee (K_{1,3}+K_1)\}$.
\end{theorem}

The following two previous results will be proved as corollaries of Theorem \ref{tNG1} in Section 3.

\begin{corollary}[Fiedler and Nikiforov \cite{Fiedler_Nikiforov}]\label{tFN}
Let $G$ be a graph on $n$ vertices. If $\rho(G)\geq n-2$, then $G$ contains a Hamilton path unless $G=K_{n-1}+K_1$.
\end{corollary}

\begin{corollary}[Lu, Liu and Tian \cite{Lu_Liu_Tian}]\label{tLLT}
Let $G$ be a connected graph on $n\geq 7$ vertices. If $\rho(G)\geq \sqrt{(n-3)^2+2}$, then $G$ contains a Hamilton path.
\end{corollary}

\begin{remark}
The original version of Corollary \ref{tLLT} (Theorem 3.4 in \cite{Lu_Liu_Tian}) uses the restriction $n\geq 5$. Note that $\rho(K_2\vee 4K_1)\approx3.3723>\sqrt{3^2+2}\approx 3.3166$ and $K_2\vee 4K_1$ contains no Hamilton path. We point out that the restriction should be $n\geq 7$.
\end{remark}

\begin{corollary}
Let $G$ be a graph on $n\geq 7$ vertices with $\delta\geq 1$. Let $G^{(1)}_n=K_1\vee (K_{n-3}+2K_1)$. If $\rho(G)\geq \rho(G^{(1)}_n)$, then $G$ contains a Hamilton path unless $G= G^{(1)}_n$.
\end{corollary}

\begin{theorem}\label{tNG2}
Let $G$ be a graph on $n$ vertices with $\delta\geq 2$. Let $G^{(2)}_n=K_2\vee (K_{n-4}+2K_1)$. (1) Suppose $n\geq 14$. If $\rho(G)\geq \rho(G^{(2)}_n)$, then $G$ contains a Hamilton cycle unless $G=G^{(2)}_n$. (2) When $n=7$, $\rho(K_3\vee 4K_1)>\rho(G^{(2)}_7)$ and $K_3\vee 4K_1$ contains no Hamilton cycle. When $n=9$, $\rho(K_4\vee 5K_1)>\rho(G^{(2)}_{9})$ and $K_4\vee 5K_1$ contains no Hamilton cycle.
\end{theorem}

Our Theorem \ref{tNG2} refines the following theorem.

\begin{theorem}[Fiedler and Nikiforov \cite{Fiedler_Nikiforov}]\label{tFN}
Let $G$ be a graph on $n$ vertices. If $\rho(G)>n-2$, then $G$ contains a Hamilton cycle unless $G=K_1\vee (K_{n-2}+K_1)$.
\end{theorem}

\section{Preliminaries}
\subsection{Some Lemmas}
Before proving the main results, we list some useful lemmas. The first one is due to Chv\'atal.

\begin{lemma}[Chv\'atal \cite{Chvatal}]\label{lechv}
Let $G$ be a graph with the degree sequence $(d_1,d_2,\ldots,d_n)$, where $d_1\leq d_2\leq\ldots\leq d_n$ and $n\geq 3$. If there is no integer $k<n/2$ such that $d_k\leq k$ and $d_{n-k}\leq n-k-1$, then $G$ contains a Hamilton cycle.
\end{lemma}

The following result is due to Ore \cite{Ore} and Bondy \cite{Bondy}, independently. Note that it supports the proof of Theorem \ref{tFN} (see Fact 1 in \cite{Fiedler_Nikiforov}).

\begin{theorem}[Ore \cite{Ore}, Bondy \cite{Bondy}]\label{leBon}
Let $G$ be a graph on $n\geq 3$ vertices and $m$ edges. If $m\geq\binom{n-1}{2}+1$, then $G$ contains a Hamilton cycle unless $G\in \mathcal{\mathcal{G}}=\{K_1\vee(K_{n-2}+K_1),K_2\vee{3K_1}\}$.
\end{theorem}

The part of the strict inequality of the following lemma is a corollary of a result due to Erd\"{o}s \cite{Erdos}. By refining the technique of Bondy  \cite{Bondy}, here we mainly characterize all the exceptional graphs when the equality holds.

\begin{lemma}\label{leNG1}
Let $G$ be a graph on $n\geq 5$ vertices and $m$ edges with $\delta\geq 2$. If $m\geq\binom{n-2}{2}+4$, then $G$ contains a Hamilton cycle unless $G\in \mathcal{\mathcal{G}}_2=\{K_2\vee (K_{n-4}+2K_1), K_3\vee 4K_1, K_2\vee (K_{1, 3}+K_1), K_1\vee K_{2,4}, K_3\vee (K_2+3K_1), K_4\vee 5K_1, K_3\vee (K_{1,4}+K_1), K_2\vee K_{2,5}, K_5\vee 6K_1\}$.
\end{lemma}
\begin{proof}
In the proof, we assume a sequence $\vec{d}$ is called a \emph{permissible graphic sequence} if there is a simple graph with degree sequence $\vec{d}$ satisfying the condition of Lemma \ref{leNG1}.

Suppose that $G$ has no Hamilton cycle and its degree sequence is $(d_1, d_2,\ldots ,d_n)$, where $d_1\leq d_2\leq \cdots \leq d_n$ and $n\geq 5$. By Lemma \ref{lechv}, there is an integer $k<\frac{n}{2}$ such that $d_k\leq k$ and $d_{n-k}\leq n-k-1$. Since $\delta\geq 2$, $k\geq d_k\geq \delta\geq 2$. Thus
\begin{equation*}
\begin{aligned}
m &=\frac{1}{2}\sum_{i=1}^n d_i \\
  &\leq \frac{1}{2}(k\cdot k+(n-2k)(n-k-1)+k(n-1))   \\
  &= \binom{n-2}{2}+4-\frac{(k-2)(2n-3k-7)}{2}.
\end{aligned}
\end{equation*}
Since $m\geq \binom{n-2}{2}+4$, $(k-2)(2n-3k-7)\leq 0$.

Assume that $(k-2)(2n-3k-7)=0$, i.e., $k=2$ or $2n-3k-7=0$. Then
$m= \binom{n-2}{2}+4$ and all inequalities in the above argument
should be equalities. If $k=2$, then $G$ is a graph with $d_1=d_2=2$, $d_3=\cdots =d_{n-2}=n-3$, and $d_{n-1}=d_n=n-1$, which implies
$G=K_2\vee (K_{n-4}+2K_1)$. If $2n=3k+7$, then $n< 14$ since $k< n/2$, and hence $n=11$, $k=5$ or $n=8$, $k=3$. The corresponding permissible graphic sequences are $(5,5,5,5,5,5,10,10,10,10,10)$ and $(3,3,3,4,4,7,7,7)$, which implies $G=K_5\vee 6K_1$ or $G=K_3\vee (3K_1+K_2)$.

Now assume $k\geq 3$ and $2n-3k-7<0$. In this case, $n\geq 2k+1\geq 7$.

If $n\geq 11$, then $2n-3k-7\geq 2n-\frac{3(n-1)}{2}-7=\frac{n-11}{2}\geq 0$. If $n=10$, then $k\leq 4$ and $2n-3k-7\geq 1$. If $n=8$, then $k\leq 3$, and hence $2n-3k-7\geq 0$. In each case, we get a contradiction.

If $n=9$, then $k\leq 4$. If $k\leq 3$, then $2n-3k-7\geq 2$, a contradiction. Now assume that $k=4$. Then $d_5\leq 4$ and $25\leq m\leq 26$. From the inequality $d_6+d_7+d_8+d_9=2m-\sum_{i=1}^{5}d_i\geq 30$, we obtain $d_8=d_9=8$ and $d_6+d_7\geq 14$. Also note that $\sum d_i=2m\geq 50$ and $\sum d_i$ is even. If $d_6=d_7=8$, then the permissible graphic sequence is $(4,4,4,4,4,8,8,8,8)$, and $G=K_4\vee 5K_1$. If $d_6=7$ and $d_7=8$, then the permissible graphic sequence is $(3,4,4,4,4,7,8,8,8)$ and $G=K_3\vee (K_{1,4}+K_1)$. If $d_6=6$ and $d_7=8$, then the permissible graphic sequence is $(4,4,4,4,4,6,8,8,8)$, hence $G=K_3\vee (K_2+K_{1,3})$ and it contains a Hamilton cycle. If $d_6=d_7=7$, then the permissible graphic sequence is $(4,4,4,4,4,7,7,8,8)$. If $v_6$ is adjacent to $v_7$, then $G$ is constructed as follows. Let $X=K_4$ and $Y=4K_1$, $x\in X$, $y_1,y_2\in Y$. $G$ is obtained from $X\vee Y$ by deleting $xy_1,xy_2$ and adding a new edge $y_1y_2$. Note that $G$ contains a Hamilton cycle. If $v_6$ is not adjacent to $v_7$, then $G=K_2\vee K_{2,5}$.

If $n=7$, then $k\leq 3$. If $k\leq 2$, then $2n-3k-7\geq 1$, a contradiction. If $k=3$, then $d_4\leq 3$ and $14\leq m\leq 15$. Since $d_5+d_6+d_7=2m-\sum_{i=1}^{4}d_i\geq 28-12=16$, we get $d_7=6$ and $d_5+d_6\geq 10$. Also note that $\sum d_i=2m\geq 28$ and $\sum d_i$ is even. If $d_5=d_6=6$, then the permissible graphic sequence is $(3,3,3,3,6,6,6)$ and $G=K_3\vee 4K_1$. If $d_5=5$ and $d_6=6$, then the permissible graphic sequence is $(2,3,3,3,5,6,6)$, and $G=K_2\vee (K_{1,3}+K_1)$. If $d_5=d_6=5$, then the permissible graphic sequence is $(3,3,3,3,5,5,6)$. If $v_5$ is not adjacent to $v_6$, then $G=K_1\vee K_{2,4}$. If $v_5$ is adjacent to $v_6$, then $G$ consists of three triangles and a quadrilateral (a cycle of length 4) sharing only one common edge, and in this case $G$ contains a Hamilton cycle.

It is clear that each graph in $\mathcal{G}_2$ contains no Hamilton cycle. The proof is complete.
\end{proof}

The following direct but interesting and useful observation is an exercise in \cite{Bondy_Murty} (see Ex 18.1.6 \cite{Bondy_Murty}).
\begin{lemma}\label{leNG2}
Let $G$ be a graph. Then $G$ contains a Hamilton path if and
only if $G\vee K_1$ contains a Hamilton cycle.
\end{lemma}

By using Lemmas \ref{leNG1} and \ref{leNG2}, we deduce the following Lemma \ref{leNG3}.

\begin{lemma}\label{leNG3}
Let $G$ be a graph on $n\geq 4$ vertices and $m$ edges with $\delta \geq 1$.
If $m\geq \binom{n-2}{2}+2$, then $G$ contains a Hamilton path
unless $G\in \mathcal{G}_1=\{K_1\vee (K_{n-3}+2K_1), K_1\vee (K_{1,3}+K_1), K_{2, 4}, K_2\vee 4K_1, K_2\vee (3K_1+K_2), K_1 \vee K_{2,5}, K_3\vee 5K_1, K_2\vee (K_{1, 4}+K_1), K_4\vee 6K_1\}$.
\end{lemma}

\begin{proof}
Now assume $G$ with $m\geq \binom{n-2}{2}+2$ contains a Hamilton path. Let $G'=G\vee K_1$. Then $|V(G')|=n+1$ and $|E(G')|\geq\binom{n-2}{2}+2+n=\binom{n-1}{2}+4$. Since $n\geq 4$, the order of $G'$ is at least $5$. By Lemma \ref{leNG1}, $G'$ contains a Hamilton cycle unless $G'\in \mathcal{G}_2=\{K_2\vee (K_{n-4}+2K_1), K_3\vee 4K_1, K_2\vee (K_{1, 3}+K_1), K_1\vee K_{2,4}, K_3\vee (K_2+3K_1), K_4\vee 5K_1, K_3\vee (K_{1,4}+K_1), K_2\vee K_{2,5}, K_5\vee 6K_1\}$. If $G'$ contains a Hamilton cycle, then by Lemma \ref{leNG2}, $G$ contains a Hamilton path. Otherwise, $G'$ contains no Hamilton cycle and $G'\in \mathcal{G}_2$. By Lemma \ref{leNG2}, $G$ contains no Hamilton path and $G\in \mathcal{G}_1$, where $\mathcal{G}_2=K_1 \vee \mathcal{G}_1\doteq\{K_1\vee G: G \in \mathcal{G}_1\}$.

The proof is complete.
\end{proof}

\begin{remark}
Recently, a similar result for the existence of Hamilton paths in connected graphs was given by Lu et al. \cite{Lu_Liu_Tian}. However, it seems to have some flaws in their original lemma (Lemma 3.2 in \cite{Lu_Liu_Tian}) and its proof (see \cite[Line 16, Page 1673]{Lu_Liu_Tian}). Compared with Lu et al.'s result, we weaken the condition and add six more exceptional graphs.
\end{remark}

Hong et al. \cite{Hong_Shu_Kang} proved the following spectral inequality for connected graphs. Nikiforov \cite{Nikiforov1} proved it for general graphs independently, and the case of equality was characterized in \cite{Zhou_Cho}.

\begin{lemma}[Nikiforov \cite{Nikiforov1}]\label{leNi}
Let $G$ be a graph on $n$ vertices and $m$ edges and let $\delta$ be the minimum degree of $G$. Then $\rho(G)\leq\frac{\delta-1}{2}+\sqrt{2m-n\delta+\frac{(\delta+1)^2}{4}}$.
\end{lemma}

The following result is also useful for us.

\begin{lemma}[Hong, Shu and Kang \cite{Hong_Shu_Kang}, Nikiforov \cite{Nikiforov1}]\label{leHSK}
For nonnegative integers $p$ and $q$ with $2q\leq p(p-1)$ and $0\leq x\leq p-1$, the function $f(x)=(x-1)/2+\sqrt{2q-px+(1+x)^2/4}$ is decreasing with respect to $x$.
\end{lemma}

The last lemma we need is a famous result on extremal graph theory due to Erd\"{o}s and Gallai. It has many generalizations and extensions. We refer the reader to, for example, \cite{Fan_Lv_Wang,Fujita_Lesniak,Woodall1,Woodall2}.

\begin{lemma}[Erd\"{o}s and Gallai \cite{Erdos_Gallai0}]\label{leEG}
Let $G$ be a graph on $n$ vertices and $m$ edges. For a given integer $k$, if $m>\frac{k(n-1)}{2}$, then $G$ contains a cycle of length at least $k+1$.
\end{lemma}

\subsection{Computation of spectral radii}
By a result of Lu et al. \cite{Lu_Liu_Tian}, the spectral radius $\rho_n^{(1)}$ of $G_n^{(1)}$ is the largest zero of the equation $x^3-(n-4)x^2-(n-1)x+2(n-4)=0$. By some routine calculation, we obtain that the spectral radius $\rho_n^{(2)}$ of $G^{(2)}_n$ is the largest zero of the equation $x^3-(n-4)x^2-(n+1)x+4(n-5)=0$. Let $G_{n,k}=K_{k}\vee (n-k)K_1$. From \cite{Nikiforov3}, we know $\rho(G_{n,k})=\big(k-1+\sqrt{4kn-(3k-1)(k+1)}\big)/2$.
The numerical results of the spectral radii of the exceptional
graphs in Lemmas \ref{leNG1} and \ref{leNG3} are shown in Tables 1 and 2, respectively.

\section{Proofs}
Before the proofs, we give some additional terminology and notation. Let $G$ be a graph, $H$ a subgraph of $G$ and $S\subset V(G)$. For a vertex $v\in V(G)$, let $N_{H}(v)$ be the set of neighbors of $v$ in $H$ and $d_{H}(v)=|N_{H}(v)|$. Moreover, let $N_{H}(S)=\cup_{v\in S} N_{H}(v)$ and $d_{H}(S)=|N_{H}(S)|$. Denote by $G[S]$ the subgraph of $G$ induced by $S$. For $G-S$, we mean the subgraph induced by $V(G)\backslash S$. Let $R$ be a cycle or path of $G$ with a given direction. For a vertex $v\in V(R)$, $v^+$ and $v^-$ are always referred to as the successor and predecessor of $v$ along the direction of $C$, respectively. For two vertices $u,v\in V(R)$, we use $\overrightarrow{R}[u,v]$ to denote the segment from $u$ to $v$ of $R$ along the direction, and $\overleftarrow{R}[u,v]$ denotes the one in the opposite direction.

\noindent{}
{\bf Proof of Theorem \ref{tNG1}.}
Suppose that $G$ has no Hamilton path. By the condition, we have
\begin{align}\label{foineq1}
\rho(G)>n-3.
\end{align}
Since $\delta\geq 1$, by a theorem due to Hong \cite{Hong}, we have
\begin{align}\label{foinHo}
\rho(G)\leq\sqrt{2m-n+1}.
\end{align}
Combining inequalities (\ref{foineq1}) with (\ref{foinHo}), we obtain $2m>n^2-5n+8$. Furthermore,
by parity, $m\geq \binom{n-2}{2}+2$. Since $G$ contains no Hamilton path,
$G\in \mathcal{G}_1=\{G^{(1)}_n, K_1\vee (K_{1,3}+K_1), K_{2, 4}, K_2\vee 4K_1,
K_2\vee (3K_1+K_2),K_1 \vee K_{2,5},K_3\vee 5K_1,K_2\vee (K_{1, 4}+K_1),K_4\vee 6K_1\}$
by Lemma \ref{leNG3}. However, by the numerical results in Table 2 and the condition $\rho(G)>n-3$,
$G$ is a graph in $\{K_1\vee (K_{n-3}+2K_1),K_2\vee 4K_1,K_1\vee (K_{1,3}+K_1)\}$. (Note that $K_{n-2}$ is a proper subgraph of
$K_1\vee (K_{n-3}+2K_1)$. By the Perron-Frobenius theorem, $\rho(K_1\vee (K_{n-3}+2K_1))>\rho(K_{n-2})=n-3$.)
{\hfill$\Box$}

\noindent{}
{\bf Proof of Corollary 1.}
Suppose that $G$ contains no Hamilton path. If $\delta(G)\geq 1$, then by the fact $\rho(G)\geq n-2$, we have $\rho(G)>n-3$. By Theorem 1, $G\in \{G^{(1)}_n,K_2\vee 4K_1,K_1\vee (K_{1,3}+K_1)\}$. However, if $G=K_2\vee 4K_1$, then $\rho(G)\approx 3.3723<4$, a contradiction. If $G=K_1\vee (K_{1,3}+K_1)$, then $\rho(G)\approx 3.1020<4$, a contradiction. If $G=G^{(1)}_n$, then we obtain $\rho(G)<n-2$ by considering the fact that the spectral radius of $G_n^{(1)}$ is the largest zero of the equation $x^3-(n-4)x^2-(n-1)x+2(n-4)=0$ (see \cite{Lu_Liu_Tian}). Hence we get a contradiction. Now we obtain $\delta(G)=0$. If $G\neq K_{n-1}+K_1$, then by the Perron-Frobenius theorem, $\rho(G)<\rho(K_{n-1}+K_1)=n-2$, contradicting the condition that
$\rho(G)\geq n-2$. Thus $G=K_{n-1}+K_1$.
{\hfill$\Box$}

\noindent{}
{\bf Proof of Corollary 2.}
Suppose $G$ contains no Hamilton path. Since $G$ is connected, $\delta\geq 1$. Meantime, $\rho(G)\geq \sqrt{(n-3)^2+2}>n-3$. By Theorem 1, $G\in \{G^{(1)}_n,K_2\vee 4K_1,K_1\vee (K_{1,3}+K_1)\}$. Since the order $n\geq 7$, $G=G^{(1)}_n$. Since the spectral radius of $G_n^{(1)}$ is the largest zero of the equation $x^3-(n-4)x^2-(n-1)x+2(n-4)=0$, it is easy to check $\rho(G^{(1)}_n)<\sqrt{(n-3)^2+2}$, a contradiction.
{\hfill$\Box$}

\noindent{}
{\bf Proof of Theorem \ref{tNG2}.}
$(1)$ Suppose there exists a graph $G$ on
$n\geq 14$ vertices satisfying $\rho(G)\geq \rho(G_n^{(2)})$, $G\neq G_n^{(2)}$
and $G$ contains no Hamilton cycle. Let $C$ be a longest cycle of $G$ of length $c$ with a given orientation. Let $R=G-V(C)$ and $V(R)=\{x_1,x_2,\ldots,x_s\}$. Without loss of generality, we assume that $d_{C}(x_1)=\max \{d_{C}(x_i)\}$. Suppose
$N_{C}(x_1)=\{y_1,y_2,\ldots,y_r\}$. Let $S=\{w_i: w_i=y_i^+, i=1,2,..,r\}$
and $H=G[V(C)\backslash S]$.

We first give the following claim, the second part of which can be proved by the technique called cycle-exchange in structural graph theory. (For the technique, see for example, \cite[Page 485]{Bondy_Murty}.) To make the context integrity, we give all the details of the proof here.
\setcounter{claim}{0}
\begin{claim}
(i) $S\cup \{x_1\}$ is an independent set.
(ii) For any two distinct vertices $w_l,w_{l'}\in S$, $d_{C}(w_l)+d_{C}(w_{l'})\leq c$.
(iii) Let $R'$ be a component of $R$. Let $u$ and $v$ be two vertices in $C$ which are neighbours of $R'$.
Then $d_{C}(u^+)+d_{C}(v^+)\leq c$.
\end{claim}
\begin{proof}
(i) If there is at least one edge in $S\cup \{x_1\}$, there will be a cycle longer than $C$, a contradiction.
(ii) Assume that $w_l$, $w_{l'}$ are
in order along the orientation. Now we choose a path $P=\overrightarrow{C}[w_l,y_{l'}]y_{l'}x_1y_l\overleftarrow{C}[y_l,w_{l'}]$. Note
that $V(P)=V(C)\cup \{x_1\}$. Set $N^-_P(w_l)=\{x^-: x\in N_P(w_l)\}$. We claim that $N^-_P(w_l)\cap N_P(w_{l'})=\emptyset$, since otherwise there is
a cycle of length $c+1$, contradicting the choice of $C$. By Claim 1(i), $w_l, w_{l'}$ are not adjacent to $x_1$. This implies that $d_P(w_l)=d_C(w_l)$ and $d_P(w_{l'})=d_C(w_{l'})$. Furthermore, $N_P(w_{l'})\subset V(P)\backslash (N^-_P(w_l)\cup \{w_{l'}\})$. Thus $d_C(w_{l'})=d_P(w_{l'})\leq c+1-(d_P(w_l)+1)$,
and we have the required inequality.
(iii) The proof is almost the same as the proof of (ii).
\end{proof}

By the Perron-Frobenius theorem,
\begin{align}\label{foineq2}
\rho(G)\geq \rho(G^{(2)}_n)>\rho{(K_{n-2})}=n-3.
\end{align}
By Lemmas \ref{leNi}, \ref{leHSK} and and the fact $\delta\geq 2$, we have
\begin{align}\label{foineq3}
\rho(G)\leq\frac{1}{2}+\sqrt{2m-2n+\frac{9}{4}}.
\end{align}
Furthermore, with inequalities (\ref{foineq2}) and (\ref{foineq3}),
we obtain
\begin{align}\label{foineq4}
2m\geq n^2-5n+10.
\end{align}
By Lemma \ref{leEG}, $G$ contains a cycle of length at least
$\lceil\frac{2m}{n-1}\rceil\geq \lceil n-4+\frac{6}{n-1}\rceil\geq n-3$. Thus $n-3\leq c\leq n-1$.

\begin{case}
$c=n-1$.
\end{case}
In this case, $V(R)=\{x_1\}$ and $r=d_{C}(x_1)=d(x_1)\geq \delta \geq 2$.

Suppose that $r>2$. Since $G$ contains no Hamilton cycle and $\delta \geq 2$, there holds
\begin{align}\label{foineq5}
3\leq r\leq \left\lfloor \frac{n-1}{2} \right\rfloor.
\end{align}

Furthermore, we obtain $m=d_C(x_1)+\sum_{i=1}^{r}d_C(w_i)+e(H)$. By Claim 1, we have $d_C(w_i)+d_C(w_{i+1})\leq n-1$ for $i=1,2,\ldots,r$, where the subscripts are taken modulo $r$. By summing up all $r$ inequalities, we have $\sum_{i=1}^{r}d_C(w_i)\leq \frac{r(n-1)}{2}$. Then
\begin{align}\label{foineq6}
m=d_C(x_1)+\sum_{i=1}^{r}d_C(w_i)+e(H)\leq r+\frac{r(n-1)}{2}+\binom{n-r-1}{2}=\frac{n^2-(r+3)n+r^2+4r+2}{2}.
\end{align}
By (\ref{foineq4}) and (\ref{foineq6}), we obtain $(r-2)n\leq r^2+4r-8$. Then
\begin{align}\label{foineq7}
n\leq\frac{r^2+4r-8}{r-2}=r+6+\frac{4}{r-2}\doteq t(r).
\end{align}
If $r=3$, then $n\leq t(3)=13$. If $r=4$, then $n\leq t(4)=12$. If $r=5$,
then $n\leq t(5)=12+1/3$, and hence $n\leq 12$. If $r=6$, then $n\leq t(6)=13$.
If $r\geq 7$, then by (\ref{foineq5}) and (\ref{foineq7}), $n\leq r+6\leq\frac{n-1}{2}+6$, and this implies
$n\leq 11$. In each case, we can get a contradiction to the assumption $n\geq 14$.

Now suppose $r=2$. There holds $\frac{n^2-5n+10}{2}\leq m\leq\frac{n^2-5n+14}{2}$.
If $m=\frac{n^2-5n+14}{2}$, then by Lemma \ref{leNG1} and $n\geq 14$,
$G=K_2\vee (K_{n-4}+2K_1)$, a contradiction. Next we denote by $\mathcal{H}=\{H_{n,2}^i:i=1,2,\ldots,t\}$ the
class of graphs obtained when the equality in (\ref{foineq6}) holds. Obviously, each $H_{n,2}^i$ has the following structure feature: \\
($a$) $x_1$ has only two neighbors $y_1,y_2$ in $H_{n,2}^i$;\\
($b$) $\{w_1,w_2,x_1\}$ is an independent set and $w_1y_1,w_2y_2\in E(G)$;\\
($c$) $V(H)$ induces a clique of $(n-3)$ vertices.

If $m=\frac{n^2-5n+14}{2}-1$ or $m=\frac{n^2-5n+14}{2}-2$, then $G$ is obtained from one graph in $\mathcal{H}$ by deleting one edge or two edges
other than $E(C)\bigcup \{x_1y_1, x_1y_2\}$.

Without loss of generality, assume $d(w_1)\leq d(w_2)$. Since $\delta\geq 2$, $d(w_1)\geq 2$.

Assume that $w_1$ has one neighbor, say for $y$, in $G$ other than
$\{y_1,y_2\}$. Noting that $d(w_1)+d(w_2)\geq n-3\geq 11$, we have
$d(w_2)\geq 6$, and this implies that there exist two vertices
$z,z'\notin \{y_1,y_2,y\}$ such that $w_2z,w_2z'\in E(G)$. Note that
there are at most two edges missing in $H$. Hence there are at least
two edges in $\{y_2z,yz',y_2z',yz\}$. Assume that $y_2z\in E(G)$. Then
let $P$ be a $(z',y)$-path such that $V(P)=V(H)\backslash \{z,y_1,y_2\}$.
(Note that by ($c$), this path always exits.) Now $C=x_1y_1w_1\overleftarrow{P}[y,z']w_2zy_2x_1$
is a Hamilton cycle in $G$, a contradiction. The other cases can be proved similarly.

Now assume that $N_G(w_1)=\{y_1,y_2\}$. Then $G$ is a proper subgraph of $G^{(2)}_n$. By the Perron-Frobenius theorem, $\rho(G)<\rho(G^{(2)}_n)$, a contradiction.
\begin{case}
$c=n-2$.
\end{case}
In this case $V(R)=\{x_1,x_2\}$, and
\begin{align}\label{foineq8}
1\leq r\leq \left\lfloor \frac{n-2}{2} \right\rfloor.
\end{align}

If $r=1$, and $|N_{C}(R)|=1$, then $G=K_1 \vee (K_{n-3}+K_2)$. We will show $\rho (G^{(2)}_n)>\rho (K_1 \vee (K_{n-3}+K_2))$ in the Appendix.

If $r=1$, and $|N_{C}(R)|=2$, assume that $N_C(x_1)=\{u\}$ and $N_C(x_2)=\{v\}$. Then $u$ and $v$ are neighbours since $d(u)\geq 2$ and $d(v)\geq 2$. Then by Claim 1(iii), at least one of $d_C(u^+)$ and $d_C(v^+)$ is at most $c/2$, say, $d_C(u^+)$, i.e., $d_C(w_1)=d_C(u^+)\leq c/2=(n-2)/2$.

In both this case and the case $r\geq 2$, by a similar argument used in Case 1, we obtain
\begin{align*}
 m&= d_C(x_1)+d_C(x_2)+\sum_{i=1}^{r}d_{C}(w_i)+e(H)+e(R)\\
&\leq 2r+\frac{r(n-2)}{2}+\binom{n-r-2}{2}+1\\
&=\frac{n^2-(r+5)n+r^2+7r+8}{2}.
\end{align*}
Recall that $m\geq(n^2-5n+10)/2$, hence $n^2-5n+10\leq n^2-(r+5)n+(r^2+7r+8)$.
It follows that $rn\leq r^2+7r-2$. Thus $n\leq r+6\leq (n-2)/2+6$ by (\ref{foineq8}), and
this implies $n\leq 10$, a contradiction.

\begin{case}
$c=n-3$.
\end{case}
In this case $V(R)=\{x_1,x_2,x_3\}$, and
\begin{align}\label{foineq9}
1\leq r\leq \left\lfloor \frac{n-3}{2} \right\rfloor.
\end{align}

When $r=1$, we obtain
\begin{align*}
 m&= d_C(x_1)+d_C(x_2)+d_C(x_3)+d_{C}(w_1)+e(H)+e(R)\\
&\leq 3+(n-4)+\binom{n-4}{2}+3\\
&=\frac{n^2-7n+24}{2}.
\end{align*}
Recall that $m\geq(n^2-5n+10)/2$, hence $n^2-5n+10\leq n^2-7n+24$.
It follows that $n\leq 7$, a contradiction.

When $r\geq 2$, we obtain
\begin{align*}
 m&= d_C(x_1)+d_C(x_2)+d_C(x_3)+\sum_{i=1}^{r}d_{C}(w_i)+e(H)+e(R)\\
&\leq 3r+\frac{r(n-3)}{2}+\binom{n-r-3}{2}+3\\
&=\frac{n^2-(r+7)n+(r^2+10r+18)}{2}.
\end{align*}
Recall that $m\geq(n^2-5n+10)/2$, hence $n^2-5n+10\leq n^2-(r+7)n+(r^2+10r+18)$.
It follows that $(r+2)n\leq r^2+10r+8$, and hence $n\leq r+8-8/(r+2)$. Since $n$
is an integer, $n\leq r+7\leq (n-3)/2+7$ by (\ref{foineq9}), and this implies
$n\leq 11$, a contradiction.

If $G=G^{(2)}_n$, then $\rho(G)\geq \rho(G^{(2)}_n)$ and $G$ contains no Hamilton cycle. The proof of the part $(1)$ of Theorem \ref{tNG2} is complete.

$(2)$ Note that $K_5\vee 6K_1$ and $K_3\vee 4K_1$ contain no Hamilton cycles. As shown by Table 1, $\rho(K_5\vee 6K_1)>\rho(G^{(2)}_{11})$ and $\rho(K_4\vee 5K_1)>\rho(G^{(2)}_9)$. The proof is complete.
{\hfill$\Box$}

\begin{remark}
It is likely that the bound of the orders of graphs in Theorem \ref{tNG2} can be sharpened to $n\geq 10$. But this may need much more complicated structural analysis of graphs or other stronger tools from spectral graph theory.
\end{remark}

\bigskip

\noindent{\bf\Large Appendix A. Comparison of $\rho(K_2\vee (K_{n-4}+2K_1))$ and $\rho(K_1 \vee (K_{n-3}+K_2))$}

\medskip

\noindent The spectral radius of $G^{(2)}_n=K_2\vee (K_{n-4}+2K_1)$ is the largest zero of the polynomial
$f_n(x)=x^3-(n-4)x^2-(n+1)x+4(n-5)$. The characteristic polynomial of the adjacency matrix
of $K_1 \vee (K_{n-3}+K_2)$ is $(x+1)^{n-3}\big(x^3-(n-3)x^2-3x+3n-11\big)$, hence
the spectral radius of $K_1 \vee (K_{n-3}+K_2)$ is the largest zero of the polynomial
$g_n(x)=x^3-(n-3)x^2-3x+3n-11$.

It is easy to see that for each $f_n(x)$, $n\geq 4$ and each $g_n(x)$, $n\geq3$, there is exactly one zero greater than $n-3$.
$$f_n(n-3+\epsilon)=\epsilon^3+(2n-5)\epsilon^2+(n^2-5n+2)\epsilon-8.$$
By some elementary calculus, when $n\geq 4$, we find $f_n(n-3+\frac{8}{n^2-5n+2})>0$ and
$f_n(n-3+\frac{8}{n^2})=\frac{8^3}{n^6}+\frac{64(2n-5)}{n^4}+\frac{16}{n^2}-\frac{40}{n}<0$.
Hence
\begin{align}\label{foineq10}
n-3+\frac{8}{n^2}<\rho (K_2\vee (K_{n-4}+2K_1))<n-3+\frac{8}{n^2-5n+2}.
\end{align}

Similarly, we obtain
$$g_n(n-3+\epsilon)=\epsilon^3+(2n-6)\epsilon^2+(n^2-6n+6)\epsilon-2.$$
By some elementary calculus, when $n\geq 5$, we find $f_n(n-3+\frac{2}{n^2-6n+6})>0$ and
$f_n(n-3+\frac{2}{n^2})=\frac{8}{n^6}+\frac{8(n-3)}{n^4}+\frac{12}{n^2}-\frac{12}{n}<0$.
Hence
\begin{align}\label{foineq11}
n-3+\frac{2}{n^2}<\rho (K_1 \vee (K_{n-3}+K_2))<n-3+\frac{2}{n^2-6n+6}.
\end{align}
When $n\geq 7$, $\frac{8}{n^2}>\frac{2}{n^2-6n+6}$ always holds. Therefore, $\rho (K_2\vee (K_{n-4}+2K_1))>\rho (K_1 \vee (K_{n-3}+K_2))$ when $n\geq 7$.

\begin{table}[htbp]
\centering
\caption{\label{spectral1} The spectral radii of exceptional graphs in Lemma \ref{leNG1}}
\smallskip
\begin{tabular}{|c|c|c|}
  \hline
    &  Graphs & Spectral radii \\\hline
  \multirow{4}{*}{$n=7$} & $K_2\vee (K_{1,3}+K_1)$ & $\frac{3+\sqrt{33}}{2}\approx 4.3723$ \\
  \cline{2-3}
   & $K_1 \vee K_{2,4}$ & $\approx 4.2182$ \\
  \cline{2-3}
   & $K_3 \vee 4K_1$ & $1+\sqrt{13}\approx 4.6056$ \\
  \cline{2-3}
   & $G_7^{(2)}$ & $\approx 4.4040$ \\\hline

  \multirow{2}{*}{$n=8$} & $K_3\vee (K_2+3K_1)$ & $\approx 5.1757$ \\
  \cline{2-3}
   & $G_8^{(2)}$ & $\frac{3+\sqrt{57}}{2}\approx 5.2749$ \\\hline

   \multirow{4}{*}{$n=9$} & $K_2 \vee K_{2,5}$ & $\approx 5.9150$ \\
  \cline{2-3}
   & $K_4 \vee 5K_1$ & $\frac{3+\sqrt{89}}{2}\approx 6.2170$ \\
  \cline{2-3}
   & $K_3 \vee (K_{1,4}+K_1)$ & $\approx 6.0322$ \\
  \cline{2-3}
   & $G_9^{(2)}$ & $\approx 6.1970$ \\\hline

  \multirow{2}{*}{$n=11$} & $K_5\vee 6K_1$ & $2+\sqrt{34}\approx 7.8310$ \\
  \cline{2-3}
  & $G_{11}^{(2)}$ & $\approx 8.1144$ \\\hline
\end{tabular}
\end{table}

\begin{table}[htbp]
\centering
\caption{\label{spectral1} The spectral radii of exceptional graphs in Lemma \ref{leNG3}}
\smallskip
\begin{tabular}{|c|c|c|}
  \hline
    &  Graphs & Spectral radii \\\hline
  \multirow{4}{*}{$n=6$} & $K_1\vee (K_{1,3}+K_1)$ & $\approx 3.1020$ \\
  \cline{2-3}
   & $K_{2,4}$ & $2\sqrt{2}\approx 2.8284$ \\
  \cline{2-3}
   & $K_2 \vee 4K_1$ & $\frac{1+\sqrt{33}}{2}\approx 3.3723$ \\
  \cline{2-3}
   & $G_6^{(1)}$ & $\approx 3.1774$ \\\hline

  \multirow{2}{*}{$n=7$} & $K_2\vee (K_2+3K_1)$ & $\approx 3.9095$ \\
  \cline{2-3}
   & $G_7^{(1)}$ & $\approx 4.1055$ \\\hline

   \multirow{4}{*}{$n=8$} & $K_1 \vee K_{2,5}$ & $\approx 4.6185$ \\
  \cline{2-3}
   & $K_3 \vee 5K_1$ & $5$ \\
  \cline{2-3}
   & $K_2 \vee (K_{1,4}+K_1)$ & $\approx 4.7903$ \\
  \cline{2-3}
   & $G_8^{(1)}$ & $\approx 5.0695$ \\\hline

  \multirow{2}{*}{$n=10$} & $K_4\vee 6K_1$ & $\frac{3+\sqrt{105}}{2}\approx 6.6235$ \\
  \cline{2-3}
  & $G_{10}^{(1)}$ & $\approx 7.0367$ \\\hline
\end{tabular}
\end{table}

\section*{Acknowledgements}
Bo Ning is supported by NSFC (No. 11271300) and the Doctorate Foundation of Northwestern Polytechnical University (cx.201326). Jun Ge is supported by NSFC (No. 11171279 and No. 11271307). The authors are grateful to the editor and referees for some helpful comments, and especially thank an anonymous referee for pointing out the inequality (2) is originally from \cite{Hong}.

\end{document}